\documentclass[12pt,a4paper]{article}
\usepackage{amsmath}
\usepackage{amssymb}
\usepackage{cases}
\usepackage{bbm}
\usepackage{mathrsfs}
\usepackage{graphicx}
\usepackage{amsfonts}
\usepackage{enumerate}
\usepackage{theorem}
\usepackage[pdfstartview=FitH]{hyperref}
\makeatletter
\def\tank#1{\protected@xdef\@thanks{\@thanks
 \protect\footnotetext[0]{#1}}}
\def\bigfoot{

 \@footnotetext}
\makeatother

\newcommand{\ea}{\end{array}}

\allowdisplaybreaks

\newtheorem{theorem}{Theorem}[section]
\newtheorem{lem}{Lemma}[section]

\newtheorem{thm}[theorem]{Theorem}
\newtheorem{cor}[theorem]{Corollary}
\newtheorem{dfn}[theorem]{Definition}

\def\beq{\begin{equation}}
\def\nneq{\end{equation}}

\def\bthm{\begin{thm}}
\def\nthm{\end{thm}}

\def\blem{\begin{lem}}
\def\nlem{\end{lem}}
\def\bprf{\begin{proof}}
\def\nprf{\end{proof}}
\def\bprop{\begin{prop}}
\def\nprop{\end{prop}}
\def\brmk{\begin{rem}}
\def\nrmk{\end{rem}}

\def\bexa{\begin{exa}}
\def\nexa{\end{exa}}
\def\bcor{\begin{cor}}
\def\ncor{\end{cor}}

{\theorembodyfont{\rmfamily}
}

\title{Strong solutions for a stochastic model of 2-D second grade fluids driven by L\'{e}vy noise}

\footnotesize{\author{
 Shijie Shang$^{1,}$\thanks{ssjln@mail.ustc.edu.cn},\ \
 Jianliang Zhai $^{1,}$\thanks{zhaijl@ustc.edu.cn},\ \
 Tusheng Zhang$^{2,}$\thanks{Tusheng.Zhang@manchester.ac.uk}\\
 {\em $^1$ School of Mathematical Sciences,}\\
 {\em University of Science and Technology of China,}\\
 {\em Hefei, 230026, China}\\
 {\em $^2$ School of Mathematics, University of Manchester,}\\
 {\em Oxford Road, Manchester, M13 9PL, UK}\\
}
\date{}
\newenvironment{proof}{\par\noindent{\bf Proof:}}{\hspace*{\fill}$\blacksquare$\par}
\begin{document}
\maketitle
\noindent \textbf{Abstract:}
We consider a stochastic model of incompressible non-Newtonian fluids of second grade on a bounded domain of $\mathbb{R}^2$ driven by L\'evy noise. Applying the variational approach, global existence and uniqueness of strong probabilistic solution is established.


\vspace{4mm}


\vspace{3mm}
\numberwithin{equation}{section}
\section{Introduction}
\indent

In petroleum industry, polymer technology and problems of liquid crystals suspensions, non-Newtonian fluids of differential type often arise (see \cite{1997-Rivlin-p911-1013}). It has attracted much attention from a theoretical point of view, since the classical theory of Newtonian fluids is unable to explain properties observed in the nature. In this article, we are interested in a special class of non-Newtonian fluids of differential type, namely the second grade fluids, which is an admissible model of slow flow fluids such as industrial fluids, slurries, polymer melts, etc.

\vskip 0.3cm

Mathematically, the stochastic models for incompressible second grade fluids are described by the following equation:
\begin{align}\label{1.a}
\left\{
\begin{aligned}
& d(u(t)-\alpha \Delta u(t))+ \Big(-\mu \Delta u(t)+{\rm curl}(u(t)-\alpha \Delta u(t))\times u(t)+\nabla\mathfrak{P}\Big)\,dt \\
& = F(u(t),t)\,dt+\sigma (u(t),t)\,dW(t) +\int_Z f(u(t-),t,z)\widetilde{N}(dzdt),\quad \rm{ in }\ \mathcal{O}\times(0,T], \\
&\begin{aligned}
& {\rm{div}}\,u=0 \quad &&\rm{in}\ \mathcal{O}\times(0,T]; \\
& u=0  &&\rm{in}\ \partial \mathcal{O}\times[0,T]; \\
& u(0)=\xi  &&\rm{in}\ \mathcal{O},
&\end{aligned}
\end{aligned}
\right.
\end{align}

\noindent where $\mathcal{O}$ is a bounded domain of $\mathbb{R}^2$ with boundary $\partial \mathcal{O}$ of class $\mathcal{C}^{3,1}$; $u=(u_1,u_2)$, $\mathfrak{P}$ represent the random velocity and modified pressure respectively; $F(u(t),t)$ is the external force; $W$ is a one-dimensional standard Brownian motion; $\widetilde{N}$ is a compensated Poisson random measure.

\vskip 0.3cm

 The second grade fluids has properties of boundedness, stability and exponential decay (see \cite{1974-Dunn-p191-252}), and has interesting connections with many other fluid models, see \cite{2004-Truesdell-p1-579,1974-Dunn-p191-252,1995-Dunn-p689-729,1979-Fosdick-p145-152,2003-Busuioc-p1119-1119,2001-Shkoller-p539-543} and references therein. For example, as shown in \cite{2002-Iftimie-p457-468}, the second grade fluids reduces to the Navier-Stokes equation(NSE) when $\alpha=0$, and it's also a good approximation of the NSE. We refer the reader to \cite{1984-Cioranescu-p178-197}, \cite{1997-Cioranescu-p317-335}, \cite{1998-Moise-p1369-1369} for a comprehensive theory of the deterministic second grade fluids.
\vskip 0.3cm

On the other hand, in recent years, introducing a jump-type noises such as L\'evy-type or Poisson-type pertubations has become extremely popular for modeling natural phenomena,
because these noises are very nice choice to reproduce the performance of some natural phenomena in real world models, such as some large moves and unpredictable events.
There is a large amount of literature on the existence and uniqueness solutions for stochastic partial differential equations(SPDEs) driven by jump-type noises. We refer the reader to \cite{2009-Dong-p1497-1524,2013-Brzezniak-p122-139,2011-Dong-p2737-2778,2012-Dong-p1150006-1150006,2015-Wang-p3363-3390,2009-Xu-p1519-1545,2015-Zhai-p2351-2392,2014-Brzezniak-p283-310,2016-Hausenblas-p763-800}.
\vskip 0.3cm

We particular mention that for the stochastic 2-D second grade fluids driven by pure jump noises, Hausenblas, Razafimandimby and Sango \cite{2013-Hausenblas-p1291-1331} obtained the global existence of a martingale solution, that is the weak solution in the probabilistic sense.
The purpose of this paper is to establish the global existence and uniqueness of strong probabilistic solutions for stochastic 2-D second grade fluids driven by general L\'evy noises. To obtain our result, we will use the so called variational approach. This approach was initiated by Pardoux \cite{1975-Pardou-p-} and Krylov and Rozovskii \cite{1981-Krylov-p1233-1277}, then further developed by many authors e.g. see \cite{1980-Gyoengy-p1-21,1982-Gyoengy-p231-254,2007-Ren-p118-152,2007-Prevot-p-,2010-Liu-p2902-2922,2014-Brzezniak-p283-310,2013-Liu-p725-755}. It has been proved to be a powerful tool to establish the well-posedness for stochastic dynamical systems with locally monotone and coercive coefficients.
However, the existing results in the literature could not cover the situation considered in this paper, because the second grade fluids do not satisfy the coercivity condition required. Under our consideration, we prove the main result by three steps: we first establish some non-trivial a priori estimates of the Galerkin approximations, then we show the limit of those approximate solutions solves the original equation by applying the monotonicity arguments, finally we prove the uniqueness of solutions.


\vskip 0.3cm
Finally, we mention several other results concerning the study of the stochastic models for two-dimensional second grade fluids driven by Wiener processes. In \cite{2010-Razafimandimby-p1-47,2012-Razafimandimby-p4251-4270}. the authors established the global existence and uniqueness of strong probabilistic solutions, and they also investigated the long time behavior of the solution and the asymptotic behavior of the solutions. Large deviations and moderate deviations for the stochastic models of second grade fluids driven by Wiener processes have been established by in \cite{2015-Zhai-p1-28} and \cite{2016-Zhai} respectively.
The exponential mixing property of the stochastic models of second grade fluids was established in \cite{2016-Wang-p196-213}.
\vskip 0.3cm

The organization of this paper is as follows. In Section 2, we introduce some preliminaries and state some lemmas which will be used later.
In Section 3, we formulate the hypotheses and state our main result.
Section 4 is devoted to the proof of the main result.

\vskip 0.3cm
Throughout this paper, $C,C_1,C_2...$ are positive constants whose value may be different from line to line.

\section{Preliminaries}

In this section, we will describe the framework in details, and introduce some functional spaces and preliminaries that are needed in the paper.
\vskip 0.2cm

Let $(\Omega,\mathcal{F},\{\mathcal{F}_t\}_{t\in[0,T]},P)$ be a given complete filtered probability space, satisfying the usual conditions.
Let $(Z,\mathcal{Z})$ be a measurable space, and $\nu$ a $\sigma$-finite measure on it. For $B\in\mathcal{Z}$ with $\nu(B)<\infty$, we write
\begin{align*}
\widetilde{N}((0,t]\times B):=N((0,t]\times B)-t\nu(B),\quad t\geq 0,
\end{align*}
for the compensated Poisson random measure on $[0,T]\times \Omega\times Z$, where $N$ is a Poisson random measure on $[0,T]\times Z$ with intensity measure $dt\times\nu(dz)$. Assume that $W$ is a one-dimensional standard Brownian motion on the probability space $(\Omega,\mathcal{F},\{\mathcal{F}_t\}_{t\in[0,T]},P)$.
\vskip 0.3cm


To formulate the equation, we introduce the following spaces.
For $p\geq 1$ and $k\in\mathbb{N}$, we denote by $L^p(\mathcal{O})$
and $W^{k,p}(\mathcal{O})$ the usual $L^p$ and Sobolev spaces over $\mathcal{O}$ respectively, and write $H^k(\mathcal{O}):=W^{k,2}(\mathcal{O})$.
Let $W^{k,p}_0(\mathcal{O})$ be
the closure in $W^{k,p}(\mathcal{O})$ of $\mathcal{C}^\infty_c(\mathcal{O})$ the space of infinitely differentiable functions with compact supports in $\mathcal{O}$, and denote $W^{k,2}_0(\mathcal{O})$ by $H_0^k(\mathcal{O})$. We equip $H^1_0(\mathcal{O})$
with the scalar product
\begin{align*}
((u,v))=\int_\mathcal{O}\nabla u\cdot\nabla vdx=\sum_{i=1}^2\int_\mathcal{O}\frac{\partial u}{\partial x_i}\frac{\partial v}{\partial x_i}dx,
\end{align*}

\noindent where $\nabla$ is the gradient operator. It is well known that the norm $\|\cdot\|$ generated by this scalar product is equivalent to the usual norm of $W^{1,2}(\mathcal{O})$.

For any vector space $X$, write $\mathbb{X}=X\times X$. Set
\begin{align*}
&\mathcal{C}=\Big\{u\in[\mathcal{C}^\infty_c(\mathcal{O})]^2 : {\rm div}\ u=0\Big\},\nonumber\\
&\mathbb{V}={\rm\ the\ closure\ of}\ \mathcal{C} {\rm\ in}\ \mathbb{H}^1(\mathcal{O}),\\
&\mathbb{H}={\rm\ the\ closure\ of}\ \mathcal{C}\ {\rm in}\ \mathbb{L}^2(\mathcal{O}).\nonumber
\end{align*}

We denote by $(\cdot,\cdot)$ and $|\cdot|$ the inner product in $\mathbb{L}^2(\mathcal{O})$(in $\mathbb{H}$) and the induced norm, respectively. The inner product and the norm of $\mathbb{H}^1_0(\mathcal{O})$
are denoted respectively by $((\cdot,\cdot))$ and $\|\cdot\|$. We endow the space $\mathbb{V}$ with the norm generated by the following inner product
\[
(u,v)_\mathbb{V}:=(u,v)+\alpha ((u,v)),\quad \text{for any } u,v\in\mathbb{V},
\]

\noindent and the norm in $\mathbb{V}$ is denoted by $|\cdot|_{\mathbb{V}}$. The $\rm Poincar\acute{e}$'s inequality implies that there exists a constant $\mathcal{P}>0$ such that the following inequalities holds
\begin{align}\label{Poincare}
(\mathcal{P}^2+\alpha)^{-1}|v|^2_\mathbb{V} \leq \|v\|^2
\leq\alpha^{-1}|v|^2_\mathbb{V},\quad \text{for any } v\in\mathbb{V}.
\end{align}

\vskip 0.2cm
We also introduce the following space
\[
\mathbb{W}=\big\{u\in\mathbb{V}: {\rm curl}(u-\alpha\Delta u)\in L^2(\mathcal{O})\big\},
\]
and endow it with the semi-norm generated by the scalar product
\begin{align}\label{W}
(u,v)_\mathbb{W}:=\big({\rm curl}(u-\alpha\Delta u),{\rm curl}(v-\alpha\Delta v)\big).
\end{align}
The semi-norm in $\mathbb{W}$ is denoted by $|\cdot|_{\mathbb{W}}$.
The following result states that this semi-norm $|\cdot|_{\mathbb{W}}$ is equivalent to the usual norm in $\mathbb{H}^3(\mathcal{O})$. The proof can be found
in \cite{1984-Cioranescu-p178-197,1997-Cioranescu-p317-335}.
\vskip 0.3cm

\begin{lem}\label{W equivalent definition}
Set

\[
\widetilde{\mathbb{W}}=\big\{v\in\mathbb{H}^3(\mathcal{O}): {\rm div}\,v=0\ {\rm and}\  v|_{\partial \mathcal{O}}=0\big\},
\]

\noindent then the following (algebraic and topological) identity holds:
\begin{align*}
\mathbb{W}=\widetilde{\mathbb{W}}.
\end{align*}

\noindent Moreover, there exists a constant $C> 0$ such that
\begin{align}\label{W-02}
    |v|_{\mathbb{H}^3(\mathcal{O})}
\leq
    C|v|_\mathbb{W},\ \ \ \forall v\in \widetilde{\mathbb{W}}.
\end{align}

\end{lem}

\vskip 0.3cm

If we identify the Hilbert space $\mathbb{V}$ with its dual space $\mathbb{V}^*$ by the Riesz representation, then we obtain a
Gelfand triple
\begin{align*}
\mathbb{W}\subset \mathbb{V}\subset\mathbb{W}^*.
\end{align*}

\noindent We denote by $\langle f,v\rangle$ the dual relation between $f\in\mathbb{W}^*$ and $v\in\mathbb{W}$ from now on. It is easy to see
\begin{align}\label{eq P4 star}
(v,w)_\mathbb{V}=\langle v,w\rangle,\ \ \ \forall\,v\in\mathbb{V},\ \ \forall\,w\in\mathbb{W}.
\end{align}

Note that the injection of $\mathbb{W}$ into $\mathbb{V}$ is compact,
thus there exists a sequence $\{e_i\}$ of elements of $\mathbb{W}$ which forms an orthonormal basis in $\mathbb{W}$, and an orthogonal system in $\mathbb{V}$, moreover this sequence verifies:
\begin{align}\label{Basis}
(v,e_i)_{\mathbb{W}}=\lambda_i(v,e_i)_{\mathbb{V}},\ \text{for any }v\in\mathbb{W},
\end{align}
where $0<\lambda_i\uparrow\infty$. Lemma 4.1 in \cite{1997-Cioranescu-p317-335} implies that
\begin{align}\label{regularity of basis}
e_i\in \mathbb{H}^4(\mathcal{O}),\ \ \forall\,i\in\mathbb{N}.
\end{align}

\vskip 0.2cm

Consider the following ``generalized Stokes equations'':
\begin{align}\label{General Stokes}
\begin{aligned}
v-\alpha \Delta v &=f\quad {\rm in}\quad\mathcal{O},\\
{\rm div}\,v &=0\quad {\rm in}\quad\mathcal{O},\\
v &=0\quad {\rm on}\quad\partial \mathcal{O}.
\end{aligned}
\end{align}

\noindent The following result can be derived from \cite{1968-Solonnikov-p269-339} and also can be found in \cite{2010-Razafimandimby-p1-47,2012-Razafimandimby-p4251-4270}.
\begin{lem}\label{Lem GS}
If $f\in\mathbb{V}$, then the system (\ref{General Stokes}) has a unique solution $v\in\mathbb{W}$, and the following relations hold
\begin{gather}
(v,g)_\mathbb{V}=(f,g),\quad \forall\, g\in \mathbb{V},\label{Eq GS-01}\\
|v|_{\mathbb{W}}\leq C|f|_{\mathbb{V}}.\label{Eq GS-02}
\end{gather}
\end{lem}

\vskip 0.3cm
Define the Stokes operator by
\begin{align}\label{Eq Stoke}
Au:=-\Pi\Delta u,\quad\forall\,u\in D(A)=\mathbb{H}^2(\mathcal{O})\cap\mathbb{V},
\end{align}
here the mapping $\Pi:\mathbb{L}^2(\mathcal{O})\longrightarrow\mathbb{H}$ is the usual Helmholtz-Leray projection.
The Helmholtz-Leray projection implies that for any $u\in \mathbb{L}^2(\mathcal{O})$, there exists a $\phi_u\in\mathbb{H}^1(\mathcal{O})$ such that
\[
u=\Pi u+\nabla\phi_u,
\]
Hence,
\begin{align}\label{curlP=curl}
{\rm curl}(\Pi v)={\rm curl}(v)\quad \forall v\in\mathbb{H}^1.
\end{align}
Therefore,
\begin{align}\label{W equivalence norm}
|u|_\mathbb{W}:=|{\rm curl}(I+\alpha A)u|.
\end{align}

\vskip 0.3cm
It follows from Lemma \ref{Lem GS} that the operator $(I+\alpha A)^{-1}$ defines an isomorphism from $\mathbb{V}$
into $\mathbb{W}$. Moreover, for any $f,g\in\mathbb{V}$, the following properties hold
\begin{align}\label{inverse operator transform}
((I+\alpha A)^{-1}f,g)_\mathbb{V}=(f,g),\nonumber\\
|(I+\alpha A)^{-1}f|_\mathbb{V}\leq C(I+\alpha A)^{-1}f|_\mathbb{W}\leq C_A |f|_{\mathbb{V}}.
\end{align}

Let $\widehat{A}:=(I+\alpha A)^{-1}A$ , then $\widehat{A}$ is a
continuous linear operator from $\mathbb{W}$ onto itself, and satisfies
\begin{align}\label{A transform 01}
(\widehat{A}u,v)_\mathbb{V}=(Au,v)=((u,v)), \quad \forall\,u\in\mathbb{W},\ v\in\mathbb{V},
\end{align}
hence
\begin{align}\label{A transform 02}
(\widehat{A}u,u)_\mathbb{V}=\|u\|, \quad\forall\,u\in\mathbb{W}.
\end{align}
%


\vskip 0.3cm

We recall the following estimates which can be found in \cite{2012-Razafimandimby-p4251-4270}.

\begin{lem}\label{Lem B}
For any $u,v,w\in\mathbb{W}$, we have
\begin{align}\label{Ineq B 01}
    |({\rm curl}(u-\alpha\Delta u)\times v,w)|
\leq
    C|u|_{\mathbb{W}}|v|_\mathbb{V}|w|_{\mathbb{W}},
\end{align}

\noindent and
\begin{align}\label{Ineq B 02}
    |({\rm curl}(u-\alpha\Delta u)\times u,w)|
\leq
    C|u|^2_\mathbb{V}|w|_{\mathbb{W}}.
\end{align}
\end{lem}

Defining the bilinear operator $\widehat{B}(\cdot\,,\cdot):\ \mathbb{W}\times\mathbb{V}\longrightarrow\mathbb{W}^*$ by
\begin{align*}
\widehat{B}(u,v):=(I+\alpha A)^{-1}\Pi\big({\rm curl}(u-\alpha \Delta u)\times v\big).
\end{align*}
We have the following consequence from the above lemma.

\begin{lem}\label{Lem-B-01}

For any $u\in\mathbb{W}$ and $v\in\mathbb{V}$, it holds that
\begin{align}\label{Eq B-01}
    |\widehat{B}(u,v)|_{\mathbb{W}^*}
\leq
    C|u|_\mathbb{W}|v|_\mathbb{V},
\end{align}

\noindent and
\begin{align}\label{Eq B-02}
|\widehat{B}(u,u)|_{\mathbb{W}^*}
\leq
    C_B|u|^2_\mathbb{V}.
\end{align}

\noindent In addition
\begin{align}\label{Eq B-03}
 \langle\widehat{B}(u,v),v\rangle=0, \quad\forall\,u, v\in\mathbb{W},
\end{align}

\noindent which implies
\begin{align}\label{Eq B-04}
 \langle\widehat{B}(u,v),w\rangle=-\langle\widehat{B}(u,w),v\rangle,\quad\forall\,u, v, w\in\mathbb{W}.
\end{align}

\end{lem}


\section{Hypotheses and the result}

In this section, we will formulate precise assumptions on coefficients and state our main results.

Let
\begin{align*}
F & :\mathbb{V}\times[0,T]\longrightarrow\mathbb{V}; \\
\sigma & :\mathbb{V}\times[0,T]\longrightarrow\mathbb{V}; \\
f & :\mathbb{V}\times [0,T]\times Z\longrightarrow\mathbb{V},
\end{align*}
be given measurable maps. We introduce the following notations:
\begin{gather*}
\widehat{F}(u,t):=(I+\alpha A)^{-1} F(u,t) ;\\
\widehat{\sigma}(u,t):=(I+\alpha A)^{-1} \sigma(u,t) ;\\
\widehat{f}(u,t,z):=(I+\alpha A)^{-1} f(u,t,z) .
\end{gather*}

Applying $(I+\alpha A)^{-1}$ to the equation (\ref{1.a}), we see that (\ref{1.a}) is equivalent to the stochastic evolution equation
\begin{align}\label{Abstract}
\left\{
 \begin{aligned}
 & du(t)+\mu \widehat{A}u(t)dt+\widehat{B}\big(u(t),u(t)\big)dt \\
 &\quad=\widehat{F}(u(t),t)+\widehat{\sigma}(u(t),t)\,dW(t) +\int_Z \widehat{f}(u(t-),t,z)\widetilde{N}(dzdt), \\
 & u(0)=\xi\quad\text{in}\ \mathbb{W}.
 \end{aligned}
\right.
\end{align}

\vskip 0.3cm

Let us formulate assumptions on coefficients. Suppose that there exists constants $C\geq 0$,  $\widetilde{\rho}\in L^1([0,T])$ and $K\in L^2([0,T])$ such that the following hold for all $u_1,u_2\in\mathbb{V}$ and $t\in [0,T]$:

\vskip 0.3cm

\noindent {\bf(S1)} (Lipschitz continuity)
\begin{align*}
&|F(u_1,t)-F(u_2,t)|_{\mathbb{V}}^2 +|\sigma(u_1,t)-\sigma(u_2,t)|_{\mathbb{V}}^2+\int_Z |f(u_1,t,z)-f(u_2,t,z)|_{\mathbb{V}}^2\,\nu(dz) \\
&\leq \widetilde{\rho}(t)|u_1-u_2|_{\mathbb{V}}^2.
\end{align*}

\noindent {\bf(S2)} (Growth condition)
\begin{align*}
|F(u,t)|_{\mathbb{V}}^2+|\sigma(u,t)|_{\mathbb{V}}^2+\int_Z |f(u,t,z)|_{\mathbb{V}}^2\,\nu(dz)\leq K(t)+C|u|_{\mathbb{V}}^2,
\end{align*}

and
\begin{align*}
\int_Z |f(u,t,z)|_{\mathbb{V}}^4\,\nu(dz)\leq K^2(t)+C|u|_{\mathbb{V}}^4.
\end{align*}

\begin{dfn}\label{Def 01}
A $\mathbb{V}$-valued c\`{a}dl\`{a}g and $\mathbb{W}$-valued weakly c\`{a}dl\`{a}g $\{\mathcal{F}_t\}$-adapted stochastic process $u$ is called a solution of the system (\ref{1.a}), if the following two conditions hold:

\noindent (1) $u\in L^2(\Omega\times[0,T];\mathbb{W})$;

\noindent (2) for any $t\in[0,T]$, the following equation holds in $\mathbb{W}^*$ $P$-a.s.:
\begin{align*}
& u(t)+\mu\int_0^t\widehat{A}u(s)\,ds+\int_0^t \widehat{B}\big(u(s),u(s)\big)\,ds \\
=&
\xi+\int_0^t\widehat{F}(u(s),s)\,ds+\int_0^t \widehat{\sigma}(u(s),s)\,dW(s)+\int_0^t \int_Z \widehat{f}(u(s-),s,z)\widetilde{N}(dzds).
\end{align*}

\end{dfn}

Now we can state the main result of this paper.

\begin{thm}\label{theorem 2.1}
Suppose that (S1) and (S2) are satisfied. Then for any $\xi\in L^4(\Omega,\mathcal{F}_0,P;\mathbb{W})$, equation (\ref{Abstract}) has a unique solution $u$. Moreover, $u\in L^4\big(\Omega,\mathcal{F},P;L^\infty([0,T];\mathbb{W})\big)$.
\end{thm}


\section{The proof of Theorem \ref{theorem 2.1}}

We divide the proof of the main result into three parts. In Section 4.1, we establish some crucial a priori estimates for the Galerkin approximations. In Section 4.2, we prove the existence. The uniqueness is proved in Section 4.3.

\subsection{Galerkin approximations}

The proof of the existence of solutions of system (\ref{Abstract}) is based on Galerkin approximations. From (\ref{Basis}) we know that $\{\sqrt{\lambda_i}e_i\}$ is an orthonormal basis of $\mathbb{V}$. Let $\mathbb{V}_n$ denote the n-dimensional subspace spanned by $\{e_1,e_2,\ldots ...,e_n\}$ in $\mathbb{W}$. Let $\Pi_n:\mathbb{W}^*\longrightarrow\mathbb{V}_n$ be defined by
\[
\Pi_ng:=\sum_{i=1}^n\lambda_i\langle g,e_i\rangle e_i,\quad \forall\, g\in\mathbb{W}^*.
\]
Clearly, $\Pi_n |_\mathbb{V}$ is just the orthogonal projection from $\mathbb{V}$ onto $\mathbb{V}_n$.

Now, for any integer $n\geq1$, we seek a solution $u_n$ to the equation

\begin{align}\label{3.1}
\left\{
\begin{aligned}
& du^n(t)= \Pi_n\big[ -\mu \widehat{A}u^n(t)-\widehat{B}\big(u^n(t),u^n(t)\big)+\widehat{F}(u^n(t),t)\big]\,dt +\Pi_n\widehat{\sigma}(u^n(t),t)dW(t)\\
&~~~~~~~~~~~ +\int_Z\Pi_n\widehat{f}(u^n(t-),t,z)\widetilde{N}(dzdt),\quad t>0, \\
& u_n(0)=\Pi_n \xi ,
\end{aligned}
\right.
\end{align}


\noindent such that
\[
u^n(t)=\sum_{j=1}^{n}u_{jn}(t)e_j, \quad t\geq 0 ,
\]
for appropriate choice of real-valued random processes $u_{jn}$.


Let
\begin{gather}
\label{3.2} \widehat{\mathcal{A}}(u,t):= -\mu\widehat{A}u-\widehat{B}(u,u)+\widehat{F}(u,t); \\
\label{4.1} \rho(u,t):=1+2C_B|u|_{\mathbb{W}}+C_A^2\widetilde{\rho}(t).
\end{gather}

\begin{lem}\label{Lemma 3.I}
Under assumptions of Theorem \ref{theorem 2.1},
for all $u,u_1,u_2\in\mathbb{W}$ and any $t\in [0,T]$, the following properties (H1), (H2) and (H3) hold:

\noindent (H1)
The map $s\longmapsto \langle\widehat{\mathcal{A}}(u_1+s u_2,t),u\rangle$ is continuous on $\mathbb{R}$.

\noindent (H2)
\begin{align*}
& 2\langle\widehat{\mathcal{A}}(u_1,t)-\widehat{\mathcal{A}}(u_2,t),u_1-u_2\rangle +|\widehat{\sigma}(u_1,t)-\widehat{\sigma}(u_2,t)|_{\mathbb{V}}^2+\int_Z |\widehat{f}(u_1,t,z)-\widehat{f}(u_2,t,z)|_{\mathbb{V}}^2\,\nu(dz) \\
\leq &\rho(u_2,t)|u_1-u_2|_{\mathbb{V}}^2.
\end{align*}

\noindent (H3)
\begin{align*}
|\widehat{\mathcal{A}}(u,t)|_{{\mathbb{W}}^*}^2\leq C\big(K(t)+|u|_{\mathbb{V}}^2+|u|_{\mathbb{V}}^4\big).
\end{align*}

\noindent If $u\in\mathbb{V}_m$ for some $m\in\mathbb{N}$, then $\widehat{\mathcal{A}}$ can satisfy (H4) for  any $\theta>0$:

\noindent (H4) (Coercivity)
\begin{align*}
\langle\widehat{\mathcal{A}}(u,t),u\rangle +\theta|u|_{\mathbb{W}}^2\leq K(t)+ C_{m,\theta}|u|_{\mathbb{V}}^2.
\end{align*}

\end{lem}

\noindent Remark: since the $\mathbb{W}$-norm is stronger than the $\mathbb{V}$-norm, (H4) in general cannot be satisfied.

\begin{proof}
Since $\widehat{A}$ is linear, $\widehat{B}$ is bilinear, $\widehat{F}$ is Lipschitz, (H1) is obvious.
By (\ref{Eq B-03}) we see that
\begin{align*}
&\langle\widehat{B}(u_1,u_1)-\widehat{B}(u_2,u_2),u_1-u_2\rangle \\
=& -\langle\widehat{B}(u_1-u_2,u_1-u_2),u_1\rangle \\
=& -\langle\widehat{B}(u_1-u_2,u_1-u_2),u_2\rangle .
\end{align*}
So, by (\ref{eq P4 star}) (\ref{A transform 01}), (\ref{inverse operator transform}), (\ref{Eq B-02}) and (S1) we have
\begin{align}
\begin{aligned}
& 2\langle\widehat{\mathcal{A}}(u_1,t)-\widehat{\mathcal{A}}(u_2,t),u_1-u_2\rangle \\
&+|\widehat{\sigma}(u_1,t)-\widehat{\sigma}(u_2,t)|_{\mathbb{V}}^2+\int_Z |\widehat{f}(u_1,t,z)-\widehat{f}(u_2,t,z)|_{\mathbb{V}}^2\,\nu(dz) \\
\leq&-2\mu\langle\widehat{A}(u_1-u_2),u_1-u_2\rangle -2\langle\widehat{B}(u_1,u_1)-\widehat{B}(u_2,u_2),u_1-u_2\rangle \\
& +|u_1-u_2|_{\mathbb{V}}^2+|\widehat{F}(u_1,t)-\widehat{F}(u_2,t)|_{\mathbb{V}}^2 \\ &+|\widehat{\sigma}(u_1,t)-\widehat{\sigma}(u_2,t)|_{\mathbb{V}}^2+\int_Z |\widehat{f}(u_1,t,z)-\widehat{f}(u_2,t,z)|_{\mathbb{V}}^2\,\nu(dz) \\
\leq &|u_1-u_2|_{\mathbb{V}}^2 + 2C_B|u_2|_{\mathbb{W}}|u_1-u_2|_{\mathbb{V}}^2 +C_A^2\widetilde{\rho}(t)|u_1-u_2|_{\mathbb{V}}^2 \\
\leq & \rho(u_2,t)|u_1-u_2|_{\mathbb{V}}^2 .
\end{aligned}
\end{align}
This proves (H2).

By (\ref{Eq B-02}) and (S2) we get
\begin{align*}
|\widehat{\mathcal{A}}(u,t)|_{{\mathbb{W}}^*}^2 \leq C\times\big(|\widehat{A}u|_{{\mathbb{W}}^*}^2+ |\widehat{B}(u,u)|_{{\mathbb{W}}^*}^2+ |\widehat{F}(u,t)|_{{\mathbb{W}}^*}^2\big)
\leq C\big(K(t)+|u|_{\mathbb{V}}^2+|u|_{\mathbb{V}}^4\big) ,
\end{align*}
which is (H3).

If $u\in\mathbb{W}_m$ for some $m\in\mathbb{N}$, there exists a constant $C_m$ such that $|u|_{\mathbb{W}}\leq C_m|u|_{\mathbb{V}}$, this will easily yields (H4).

\end{proof}

By Lemma \ref{Lemma 3.I}, according to Theorem 2 of \cite{1980-Gyoengy-p1-21}, equation (\ref{3.1}) has a unique global c\`{a}dl\`{a}g solution $u^n$ that satisfies the following integral equation:

\begin{align}\label{6.1}
\begin{aligned}
u^n(t)=&\Pi_n\xi +\int_0^t\Pi_n\widehat{\mathcal{A}}(u^n(s),s)\,ds +\int_0^t\Pi_n\widehat{\sigma}(u^n(s),s)\,dW(s)\\
&+ \int_0^t\int_Z\Pi_n\widehat{f}(u^n(s-),s,z)\widetilde{N}(dzds) ,
\end{aligned}
\end{align}
for any $t\in[0,T]$

In order to construct solutions to equation (\ref{Abstract}), we need to establish some a priori estimates for $u^n$.

\begin{lem}\label{Lemma 7.I}
Under assumptions of Theorem \ref{theorem 2.1}, we have
\begin{align}\label{7.1.b}
\sup_{n\in\mathbb{N}}E\Big(\sup_{t\in[0,T]}|u^n(t)|^4_{\mathbb{W}}\Big)\leq C\times\Big( \int_0^T K^2(s)\,ds +E|\xi|_{\mathbb{W}}^{4}\Big).
\end{align}
where $K(s), s\leq T$ is the function appeared in the condition (S2).

\end{lem}

\begin{proof}
For any given $n\in\mathbb{N}$, $M>0$, we define
\begin{align}
\tau^n_M:=\inf\{t\geq0: |u^n(t)|_{\mathbb{W}}\geq M\}\wedge T.
\end{align}
Since the solution $u^n$ is c\`{a}dl\`{a}g in $\mathbb{W}$ and $\{\mathcal{F}_t\}$-adapted, $\tau^n_M$ is a stopping time, and moreover, $\tau^n_M \uparrow T$ $P$-a.s. as $M\rightarrow\infty$. Note that $|u^n(t)|_{\mathbb{W}}\leq M$, for $t<\tau^n_M$, so there exists a constant $C$ such that $|u^n(t)|_{\mathbb{V}}\leq CM$ whenever $t<\tau^n_M$.

%
By (\ref{eq P4 star}) and (\ref{3.1}) for $i=1,\ldots,n$, we have
\begin{align}\label{14.1}
\begin{aligned}
& d(u^n(t),e_i)_{\mathbb{V}}=d\langle u^n(t),e_i\rangle \\
=& \langle\Pi_n\widehat{\mathcal{A}}(u^n(t),t),e_i\rangle\,dt+\langle\Pi_n\widehat{\sigma}(u^n(t),t),e_i\rangle\,dW(t) +\int_Z\langle\Pi_n\widehat{f}(u^n(t-),t,z),e_i\rangle\widetilde{N}(dzdt) \\
=&(\widehat{\mathcal{A}}(u^n(t),t),e_i)_{\mathbb{V}}\,dt+(\Pi_n\widehat{\sigma}(u^n(t),t),e_i)_{\mathbb{V}}\,dW(t) +\int_Z(\Pi_n\widehat{f}(u^n(t-),t,z),e_i)_{\mathbb{V}}\widetilde{N}(dzdt) .
\end{aligned}
\end{align}

\noindent By (\ref{regularity of basis}) and Lemma \ref{Lem GS}, we know that $\widehat{\mathcal{A}}(u^n(t),t)\in\mathbb{W}$, so multiplying both sides of the equation (\ref{14.1}) by $\lambda_i$, we can use (\ref{Basis}) to obtain
\begin{align*}
d(u^n(t),e_i)_{\mathbb{W}}=(\widehat{\mathcal{A}}(u^n(t),t),e_i)_{\mathbb{W}}\,dt +(\Pi_n\widehat{\sigma}(u^n(t),t),e_i)_{\mathbb{W}}\,dW(t) +\int_Z (\Pi_n\widehat{f}(u^n(t-),t,z),e_i)_{\mathbb{W}}\,\widetilde{N}(dzdt) ,
\end{align*}

\noindent for $i=1,\ldots,n$. Applying It\^{o} formula to $(u^n(t),e_i)_{\mathbb{W}}^2$ and then summing over $i$ from $1$ to $n$ yields
\begin{align}\label{16.1}
\begin{aligned}
|u^n(t)|_{\mathbb{W}}^2=& |u^n(0)|_{\mathbb{W}}^2 +2\int_0^t(\widehat{\mathcal{A}}(u^n(s),s),u^n(s))_{\mathbb{W}}\,ds +\int_0^t|\Pi_n\widehat{\sigma}(u^n(s),s)|_{\mathbb{W}}^2\,ds \\ &+2\int_0^t(\Pi_n\widehat{\sigma}(u^n(s),s),u^n(s))_{\mathbb{W}}\,dW(s)  +2\int_0^t\int_Z (\Pi_n\widehat{f}(u^n(s-),s,z),u^n(s-))_{\mathbb{W}}\widetilde{N}(dzds) \\ &+\int_0^t\int_Z|\Pi_n\widehat{f}(u^n(s-),s,z)|_{\mathbb{W}}^2\,N(dzds) .
\end{aligned}
\end{align}

\noindent By (\ref{curlP=curl}) we have
\begin{align}\label{curlPcurl}
\begin{aligned}
&\Big({\rm curl}\big[\Pi\big({\rm curl}\big(u^n(s)-\alpha\Delta u^n(s)\big)\times u^n(s)\big)\big],{\rm curl}\big(u^n(s)-\alpha\Delta u^n(s)\big)\Big) \\
=&\Big({\rm curl}\big({\rm curl}\big(u^n(s)-\alpha\Delta u^n(s)\big)\times u^n(s)\big),{\rm curl}\big(u^n(s)-\alpha\Delta u^n(s)\big)\Big),
\end{aligned}
\end{align}

\noindent Noticing that (see (4.44), (4.45) and (4.46) in \cite{2010-Razafimandimby-p1-47})
\begin{align}\label{curlcurl=0}
\Big({\rm curl}\big({\rm curl}\big(u^n(s)-\alpha\Delta u^n(s)\big)\times u^n(s)\big),{\rm curl}\big(u^n(s)-\alpha\Delta u^n(s)\big)\Big)=0,
\end{align}

%

\noindent by (\ref{W equivalence norm}), (\ref{3.2}), (\ref{curlPcurl}) and (\ref{curlcurl=0}) we get
\begin{align}\label{16.2}
\begin{aligned}
& (\widehat{\mathcal{A}}(u^n(s),s),u^n(s))_{\mathbb{W}}
= \Big({\rm curl}\big[(I+\alpha A)\widehat{\mathcal{A}}(u^n(s),s)\big],{\rm curl}\big(u^n(s)-\alpha\Delta u^n(s)\big)\Big) \\
=& (\widehat{F}(u^n(s),s),u^n(s))_{\mathbb{W}}+ \mu\Big({\rm curl}(\Delta u^n(s)),{\rm curl}\big(u^n(s)-\alpha\Delta u^n(s)\big)\Big) \\
& -\Big({\rm curl}\big[\Pi\big({\rm curl}\big(u^n(s)-\alpha\Delta u^n(s)\big)\times u^n(s)\big)\big],{\rm curl}\big(u^n(s)-\alpha\Delta u^n(s)\big)\Big) \\
=& (\widehat{F}(u^n(s),s),u^n(s))_{\mathbb{W}}-\frac{\mu}{\alpha}\big|u^n(s)\big|_{\mathbb{W}}^2 + \frac{\mu}{\alpha}\Big({\rm curl}\big(u^n(s)\big),{\rm curl}\big(u^n(s)-\alpha\Delta u^n(s)\big)\Big) .
\end{aligned}
\end{align}

\noindent Substituting the above equality into (\ref{16.1}) we have
\begin{align}\label{17.1}
\begin{aligned}
& |u^n(t)|_{\mathbb{W}}^2+ \frac{2\mu}{\alpha}\int_0^t|u^n(s)|_{\mathbb{W}}^2\,ds \\
=& |u^n(0)|_{\mathbb{W}}^2 +\frac{2\mu}{\alpha}\int_0^t\Big({\rm curl}\big(u^n(s)\big),{\rm curl}\big(u^n(s)-\alpha\Delta u^n(s)\big)\Big)\,ds \\
& +2\int_0^t(\widehat{F}(u^n(s),s),u^n(s))_{\mathbb{W}}\,ds  +\int_0^t|\Pi_n\widehat{\sigma}(u^n(s),s)|_{\mathbb{W}}^2\,ds \\ &+2\int_0^t(\Pi_n\widehat{\sigma}(u^n(s),s),u^n(s))_{\mathbb{W}}\,dW(s)  +2\int_0^t\int_Z (\Pi_n\widehat{f}(u^n(s-),s,z),u^n(s-))_{\mathbb{W}}\widetilde{N}(dzds) \\ &+\int_0^t\int_Z|\Pi_n\widehat{f}(u^n(s-),s,z)|_{\mathbb{W}}^2\,N(dzds) .
\end{aligned}
\end{align}

\noindent Applying the It\^{o} formula (cf.\cite{1982-Metivier-p-}) to (\ref{17.1}) we have

\begin{align}\label{21.6.1}
\begin{aligned}
& |u^n(t)|_{\mathbb{W}}^4+ \frac{4\mu}{\alpha}\int_0^t|u^n(s)|_{\mathbb{W}}^4\,ds \\
=& |u^n(0)|_{\mathbb{W}}^4 +\frac{4\mu}{\alpha}\int_0^t|u^n(s)|_{\mathbb{W}}^2\Big({\rm curl}\big(u^n(s)\big),{\rm curl}\big(u^n(s)-\alpha\Delta u^n(s)\big)\Big)\,ds \\
& +4\int_0^t|u^n(s)|_{\mathbb{W}}^2(\widehat{F}(u^n(s),s),u^n(s))_{\mathbb{W}}\,ds \\
& +2\int_0^t|u^n(s)|_{\mathbb{W}}^2|\Pi_n\widehat{\sigma}(u^n(s),s)|_{\mathbb{W}}^2\,ds +4\int_0^t(\Pi_n\widehat{\sigma}(u^n(s),s),u^n(s))_{\mathbb{W}}^2\,ds  \\
&+4\int_0^t|u^n(s)|_{\mathbb{W}}^2(\Pi_n\widehat{\sigma}(u^n(s),s),u^n(s))_{\mathbb{W}}\,dW(s)  \\
&+4\int_0^t\int_Z|u^n(s-)|_{\mathbb{W}}^2(\Pi_n\widehat{f}(u^n(s-),s,z),u^n(s-))_{\mathbb{W}}\widetilde{N}(dzds) \\ &+4\int_0^t\int_Z(\Pi_n\widehat{f}(u^n(s-),s,z),u^n(s-))_{\mathbb{W}}^2\,N(dzds) \\
&+2\int_0^t\int_Z |u^n(s-)|_{\mathbb{W}}^2|\Pi_n\widehat{f}(u^n(s-),s,z)|_{\mathbb{W}}^2\,N(dzds) \\
&+\int_0^t\int_Z |\Pi_n\widehat{f}(u^n(s-),s,z)|_{\mathbb{W}}^4\,N(dzds) \\
&+4\int_0^t\int_Z (\Pi_n\widehat{f}(u^n(s-),s,z),u^n(s-))_{\mathbb{W}}|\Pi_n\widehat{f}(u^n(s-),s,z)|_{\mathbb{W}}^2\,N(dzds).
\end{aligned}
\end{align}

\noindent Set
\begin{align*}
Y^n(r):=\sup_{0\leq t\leq r\wedge\tau^n_M}|u^n(t)|_{\mathbb{W}}^4 .
\end{align*}

\noindent Taking the sup over $t\leq r\wedge\tau^n_M$, then taking expectations we get
\begin{align}\label{21.6.2}
\begin{aligned}
& EY^n(r) +\frac{4\mu}{\alpha} E \int_0^{r\wedge\tau^n_M}|u^n(s)|_{\mathbb{W}}^4\,ds \\
\leq & E|u^n(0)|_{\mathbb{W}}^4 +I_1+I_2+I_3+I_4 ,
\end{aligned}
\end{align}
where
\begin{align*}
&\begin{aligned}
  I_1:=& \frac{4\mu}{\alpha}E\int_0^{r\wedge\tau^n_M}|u^n(s)|_{\mathbb{W}}^2\Big|\Big({\rm curl}\big(u^n(s)\big),{\rm curl}\big(u^n(s)-\alpha\Delta u^n(s)\big)\Big)\Big|\,ds \\
& +4E\int_0^{r\wedge\tau^n_M}|u^n(s)|_{\mathbb{W}}^2\big|(\widehat{F}(u^n(s),s),u^n(s))_{\mathbb{W}}\big|\,ds \\
& +2E\int_0^{r\wedge\tau^n_M}|u^n(s)|_{\mathbb{W}}^2|\Pi_n\widehat{\sigma}(u^n(s),s)|_{\mathbb{W}}^2\,ds +4E\int_0^{r\wedge\tau^n_M}(\Pi_n\widehat{\sigma}(u^n(s),s),u^n(s))_{\mathbb{W}}^2\,ds  ;
&\end{aligned}
\\
&\begin{aligned}
    &I_2:=4E\sup_{0\leq t\leq r\wedge\tau^n_M}\Big|\int_0^t|u^n(s)|_{\mathbb{W}}^2(\Pi_n\widehat{\sigma}(u^n(s),s),u^n(s))_{\mathbb{W}}\,dW(s) \Big|;
&\end{aligned}
\\
&\begin{aligned}
    I_3:= 4E\sup_{0\leq t\leq r\wedge\tau^n_M}\Big|\int_0^t\int_Z|u^n(s-)|_{\mathbb{W}}^2(\Pi_n\widehat{f}(u^n(s-),s,z),u^n(s-))_{\mathbb{W}}\widetilde{N}(dzds)\Big| ;
&\end{aligned}
\\
&\begin{aligned}
I_4:=E\int_0^{r\wedge\tau^n_M}\int_Z\big|\mathcal{H}(u^n(s-),s,z)\big|\,N(dzds) ,
&\end{aligned}
\end{align*}

\noindent where
\begin{align*}
\mathcal{H}(u^n(s-),s,z):=& 4(\Pi_n\widehat{f}(u^n(s-),s,z),u^n(s-))_{\mathbb{W}}^2 +2|u^n(s-)|_{\mathbb{W}}^2|\Pi_n\widehat{f}(u^n(s-),s,z)|_{\mathbb{W}}^2 \\
& +|\Pi_n\widehat{f}(u^n(s-),s,z)|_{\mathbb{W}}^4 +4(\Pi_n\widehat{f}(u^n(s-),s,z),u^n(s-))_{\mathbb{W}}|\Pi_n\widehat{f}(u^n(s-),s,z)|_{\mathbb{W}}^2 .
\end{align*}

We now estimate terms $I_1,I_2,I_3,I_4$.
By the condition (S2) and the inequalities
\[
|\widehat{F}(u^n(s),s)|_{\mathbb{W}}+|\widehat{\sigma}(u^n(s),s)|_{\mathbb{W}}\leq C(|F(u^n(s),s)|_{\mathbb{V}}+|\sigma(u^n(s),s)|_{\mathbb{V}}),
\]
\[
|u^n(s)|_{\mathbb{W}}^3|F(u^n(s),s)|_{\mathbb{V}}\leq \frac{1}{2}|u^n(s)|_{\mathbb{W}}^4+\frac{1}{2}|u^n(s)|_{\mathbb{W}}^2|F(u^n(s),s)|^2_{\mathbb{V}},
\]
\[
\big|{\rm curl}(v)\big|^2\leq \frac{2}{\alpha}|v|_{\mathbb{V}}^2,\quad \text{for any } v\in\mathbb{V} ,
\]
we have
\begin{align}\label{21.7.1}
\begin{aligned}
I_1:\leq & C E\int_0^{r\wedge\tau^n_M}|u^n(s)|_{\mathbb{W}}^3|u^n(s)|_{\mathbb{V}}\,ds
+C E\int_0^{r\wedge\tau^n_M}|u^n(s)|_{\mathbb{W}}^3|F(u^n(s),s)|_{\mathbb{V}}\,ds \\
& +6E\int_0^{r\wedge\tau^n_M}|u^n(s)|_{\mathbb{W}}^2|\widehat{\sigma}(u^n(s),s)|_{\mathbb{W}}^2\,ds \\
\leq & CE\int_0^{r\wedge\tau^n_M}|u^n(s)|_{\mathbb{W}}^4\,ds
+C E\int_0^{r\wedge\tau^n_M}|u^n(s)|_{\mathbb{W}}^2|F(u^n(s),s)|_{\mathbb{V}}^2\,ds \\
& +CE\int_0^{r\wedge\tau^n_M}|u^n(s)|_{\mathbb{W}}^2|\sigma(u^n(s),s)|_{\mathbb{V}}^2\,ds  \\
\leq & CE\int_0^{r\wedge\tau^n_M}|u^n(s)|_{\mathbb{W}}^4\,ds + C\int_0^T K^2(s)\,ds .
\end{aligned}
\end{align}

\noindent For $I_2$ and $I_3$, we apply BDG inequality (cf. Theorem 1 of \cite{1986-Ichikawa-p329-339}) and Young inequality

\begin{align}\label{21.8.1}
\begin{aligned}
I_2
\leq & C E\Big(\int_0^{r\wedge\tau^n_M}|u^n(s)|_{\mathbb{W}}^4(\Pi_n\widehat{\sigma}(u^n(s),s),u^n(s))_{\mathbb{W}}^2\,ds\Big)^{\frac{1}{2}} \\
\leq & C E\bigg[\sup_{0\leq s\leq r\wedge\tau^n_M}|u^n(s)|_{\mathbb{W}}^2 \times\Big(\int_0^{r\wedge\tau^n_M}(\Pi_n\widehat{\sigma}(u^n(s),s),u^n(s))_{\mathbb{W}}^2\,ds\Big)^{\frac{1}{2}}\bigg] \\
\leq & C\varepsilon_2 E\sup_{0\leq s\leq r\wedge\tau^n_M}|u^n(s)|_{\mathbb{W}}^4 +\frac{C}{\varepsilon_2}E\int_0^{r\wedge\tau^n_M}(\Pi_n\widehat{\sigma}(u^n(s),s),u^n(s))_{\mathbb{W}}^2\,ds \\
\leq & C\varepsilon_2 EY^n(r) +\frac{C}{\varepsilon_2}E\int_0^{r\wedge\tau^n_M}|u^n(s)|_{\mathbb{W}}^2|\sigma(u^n(s),s)|^2_{\mathbb{V}}\,ds \\
\leq & C\varepsilon_2 EY^n(r) +\frac{C}{\varepsilon_2}\int_0^T K^2(s)\,ds +\frac{C}{\varepsilon_2}E\int_0^{r\wedge\tau^n_M}|u^n(s)|_{\mathbb{W}}^4\,ds ;
\end{aligned}
\end{align}

\begin{align}\label{21.8.2}
\begin{aligned}
I_3
\leq & CE\Big(\int_0^{r\wedge\tau^n_M}\int_Z|u^n(s)|_{\mathbb{W}}^4(\Pi_n\widehat{f}(u^n(s-),s,z),u^n(s))_{\mathbb{W}}^2\,\nu(dz)ds\Big)^{\frac{1}{2}} \\
\leq & C E\bigg[\sup_{0\leq s\leq r\wedge\tau^n_M}|u^n(s)|_{\mathbb{W}}^2 \times\Big(\int_0^{r\wedge\tau^n_M}\int_Z(\Pi_n\widehat{f}(u^n(s-),s,z),u^n(s))_{\mathbb{W}}^2\,\nu(dz)ds\Big)^{\frac{1}{2}}\bigg] \\
\leq & C\varepsilon_3 E\sup_{0\leq s\leq r\wedge\tau^n_M}|u^n(s)|_{\mathbb{W}}^4 +\frac{C}{\varepsilon_3}E\int_0^{r\wedge\tau^n_M}\int_Z(\Pi_n\widehat{f}(u^n(s-),s,z),u^n(s))_{\mathbb{W}}^2\,\nu(dz)ds \\
\leq & C\varepsilon_3 EY^n(r) +\frac{C}{\varepsilon_3}E\int_0^{r\wedge\tau^n_M}\int_Z|u^n(s)|_{\mathbb{W}}^2|f(u^n(s),s,z)|_{\mathbb{V}}^2\,\nu(dz)ds \\
\leq & C\varepsilon_3 EY^n(r) +\frac{C}{\varepsilon_3}E\int_0^{r\wedge\tau^n_M}|u^n(s)|_{\mathbb{W}}^2(K(s)+|u^n(s)|^2_{\mathbb{V}})\,ds \\
\leq & C\varepsilon_3 EY^n(r) +\frac{C}{\varepsilon_3}\int_0^T K^2(s)\,ds +\frac{C}{\varepsilon_3}E\int_0^{r\wedge\tau^n_M}|u^n(s)|_{\mathbb{W}}^4\,ds .
\end{aligned}
\end{align}

\noindent By a straightforward calculation we see that
\begin{align*}
\big|\mathcal{H}(u^n(s-),s,z)\big|
\leq  C|u^n(s-)|_{\mathbb{W}}^2|f(u^n(s-),s,z)|_{\mathbb{V}}^2 +C|f(u^n(s-),s,z)|_{\mathbb{V}}^4 ,
\end{align*}

\noindent therefore, by (S2) we have
\begin{align}\label{21.9.1}
\begin{aligned}
I_4=& E\int_0^{r\wedge\tau^n_M}\int_Z\big|\mathcal{H}(u^n(s),s,z)\big|\,\nu(dz)ds \\
\leq & CE\int_0^{r\wedge\tau^n_M}\int_Z|u^n(s)|_{\mathbb{W}}^2|f(u^n(s),s,z)|_{\mathbb{V}}^2\,\nu(dz)ds
 +CE\int_0^{r\wedge\tau^n_M}\int_Z|f(u^n(s),s,z)|_{\mathbb{V}}^4\,\nu(dz)ds \\
\leq & CE\int_0^{r\wedge\tau^n_M}|u^n(s)|_{\mathbb{W}}^2(K(s)+C|u^n(s)|_{\mathbb{V}}^2)\,ds + CE\int_0^{r\wedge\tau^n_M}(K^2(s)+C|u^n(s)|_{\mathbb{V}}^4)\,ds \\
\leq & CE\int_0^{r\wedge\tau^n_M}|u^n(s)|_{\mathbb{W}}^4\,ds + C\int_0^T K^2(s)\,ds .
\end{aligned}
\end{align}

\noindent Substituting (\ref{21.7.1})-(\ref{21.9.1}) into (\ref{21.6.2}), then choosing sufficiently small $\varepsilon_2,\varepsilon_3$ we obtain
\begin{align}\label{21.9.2}
\begin{aligned}
EY^n(r) \leq C E|u^n(0)|_{\mathbb{W}}^4 +C\int_0^T K^2(s)\,ds + C E\int_0^r Y^n(s)\,ds .
\end{aligned}
\end{align}

\noindent Applying Gronwall inequality we have
\begin{align}\label{21.9.3}
 E\sup_{0\leq t\leq r\wedge\tau^n_M}|u^n(t)|_{\mathbb{W}}^4
\leq  C\times\Big(\int_0^T K^2(s)\,ds +E|\xi|_{\mathbb{W}}^4\Big) .
\end{align}

\noindent Therefore, letting $M\rightarrow\infty$, by the Fatou lemma we get
\begin{align}\label{21.9.4}
 E\sup_{0\leq t\leq T}|u^n(t)|_{\mathbb{W}}^4
\leq  C\times\Big(\int_0^T K^2(s)\,ds +E|\xi|_{\mathbb{W}}^4\Big) .
\end{align}
The proof of Lemma \ref{Lemma 7.I} is complete.

\end{proof}

\subsection{Existence of solutions}

\begin{lem}\label{Lemma 22.I}
Under assumptions of Theorem \ref{theorem 2.1}, there exists a subsequence $(n_k)$ and an element $\overline{u}\in L^4(\Omega;L^{\infty}([0,T];\mathbb{W}))$ such that
\begin{align*}
&(i)\  u^{n_k}\rightharpoonup \overline{u} \ \text{weakly in}\ L^2(\Omega\times[0,T];\mathbb{V})\ \text{and in}\ L^2(\Omega;L^{\infty}([0,T];\mathbb{V}))\ \text{under weak-* topology;}\\
&(ii)\  \widehat{A}u^{n_k}(s)\rightharpoonup \widehat{A}(s)\overline{u} \ \text{ in}\ L^2(\Omega\times[0,T];\mathbb{W}^*)\ \text{under weak-* topology;} \\
&(iii)\  \widehat{B}(u^{n_k}(s),u^{n_k}(s))\rightharpoonup \widehat{B}^*(s) \ \text{ in}\ L^2(\Omega\times[0,T];\mathbb{W}^*)\ \text{under weak-* topology;} \\
&(iv)\  \widehat{F}(u^{n_k}(s),s)\rightharpoonup \widehat{F}^*(s) \ \text{ in}\ L^2(\Omega\times[0,T];\mathbb{W}^*)\ \text{under weak-* topology;} \\
&(v)\   \widehat{\sigma}(u^{n_k}(s),s)\rightharpoonup \widehat{\sigma}^*(s) \ \text{ weakly in}\ L^2(\Omega\times[0,T];\mathbb{V}),\\
&~~~ \ \text{and}\  \widehat{\sigma}^*(s) \text{ is } \{\mathcal{F}_t\} \text{-progressively measurable, moreover}\\
& \ \begin{aligned}
\quad \int_0^\cdot\widehat{\sigma}(u^{n_k}(s),s)\,dW(s)\rightharpoonup\int_0^\cdot\widehat{\sigma}^*(s)\,dW(s)
\ \text{ in}\ L^{\infty}([0,T];L^2(\Omega,\mathcal{F}_T;\mathbb{V}))\ \text{under weak-* topology;}
& \ \end{aligned} \\
&(vi)\  \widehat{f}(u^{n_k}(s-),s,z)\rightharpoonup \widehat{f}^*(s,z) \ \text{ weakly in}\ L^2(\Omega\times[0,T]\times Z;P\times \nu\times dt;\mathbb{V}),\\
&~~~\text{ and } \widehat{f}^*(s,z) \text{ is } \{\mathcal{F}_t\} \text{-predictable, moreover in}\ L^{\infty}([0,T];L^2(\Omega,\mathcal{F}_T;\mathbb{V}))\ \text{under weak-* topology, } \\
& \ \begin{aligned}
\quad\int_0^\cdot\int_Z\widehat{f}(u^{n_k}(s-),s)\widetilde{N}(dzds)\rightharpoonup \int_0^\cdot\int_Z\widehat{f}^*(s,z)\widetilde{N}(dzds) .
& \ \end{aligned}
\end{align*}

\end{lem}

\begin{proof}
(i) is a direct corollary of lemma \ref{Lemma 7.I}. Since a bounded linear operator between two Banach spaces is trivially weakly continuous, so (ii) is derived by (i).
By (\ref{Eq B-02}) and Lemma \ref{Lemma 7.I} we have
\[E\int_0^T|\widehat{B}(u^{n_k}(s),u^{n_k}(s))|_{\mathbb{W}^*}^2\,ds\leq CE\int_0^T|u^{n_k}|^4_\mathbb{V}\,ds\leq C. \]
Therefore the claim (iii) holds.
By (S2) and Lemma \ref{Lemma 7.I} we know that (iv) also holds.
The first claim of (v) holds for the the same reason as (iv).
The second claim of (v) follows because
\begin{align*}
\Big\|\int_0^{\cdot}\widehat{\sigma}(u^{n_k}(s),s)\,dW(s)\Big\|_{L^{\infty}([0,T];L^2(\Omega,\mathcal{F}_T;\mathbb{V}))}^2 =\big\|\widehat{\sigma}(u^{n_k}(\cdot),\cdot)\big\|_{L^2(\Omega\times[0,T];\mathbb{V})}^2,
\end{align*}
and  a bounded linear operator between two Banach spaces is weakly continuous. (vi) holds for the same reason as (v).
\end{proof}

\begin{thm}\label{theorem existence}
Suppose the assumptions of Theorem \ref{theorem 2.1} are in place, then there exists a solution to equation (\ref{Abstract}).
\end{thm}

\begin{proof}
Let $\overline{u}$ be the process constructed in Lemma \ref{Lemma 22.I}. Set
\begin{align}\label{22.1}
\widehat{\mathcal{A}}^*(s):= -\mu\widehat{A}\overline{u}(s)-\widehat{B}^*(s)+\widehat{F}^*(s) ,
\end{align}
then we have
\begin{align*}
\widehat{\mathcal{A}}(u^{n_k}(s),s)\rightharpoonup \widehat{\mathcal{A}}^*(s)\ \text{ in}\ L^2(\Omega\times[0,T];\mathbb{W}^*)\ \text{under weak-* topology} .
\end{align*}

\noindent Let us define a $\mathbb{W}^*$-valued process $u$ by
\begin{align*}
 u(t):= u_0 + \int _0^t\widehat{\mathcal{A}}^*(s)\,ds +\int_0^t\widehat{\sigma}^*(s)\,dW(s) +\int_0^t\int_Z\widehat{f}^*(s,z)\widetilde{N}(dzds),\quad t\in [0,T].
\end{align*}

\noindent We first show that $u=\overline{u}$ $dt\otimes$ $P$-a.e..
For all $w\in \bigcup_{n=1}^{\infty}\mathbb{W}_n$, $\varphi\in L^{\infty}([0,T]\times\Omega)$
\begin{align*}
& E\int_0^T\langle\overline{u}(t),\varphi(t)w\rangle\,dt=\lim_{k\rightarrow\infty}E\int_0^T\langle u^{n_k}(t),\varphi(t)w\rangle\,dt \\
=&\lim_{k\rightarrow\infty}E\int_0^T\langle u^{n_k}(0),\varphi(t)w\rangle\,dt +\lim_{k\rightarrow\infty}E\int_0^T\int_0^t\langle\Pi_{n_k}\widehat{\mathcal{A}}(u^{n_k}(s),s),\varphi(t)w\rangle\,dsdt \\
&+\lim_{k\rightarrow\infty}E\int_0^T\langle\int_0^t\Pi_{n_k}\widehat{\sigma}(u^{n_k}(s),s)\,dW(s),\varphi(t)w\rangle\,dt\\
&+\lim_{k\rightarrow\infty}E\int_0^T\langle\int_0^t\int_Z\Pi_{n_k}\widehat{f}(u^{n_k}(s-),s,z)\widetilde{N}(dzds),\varphi(t)w\rangle\,dt\\
=&\lim_{k\rightarrow\infty}E\Big((u^{n_k}(0),w)_{\mathbb{V}}\int_0^T\varphi(t)\,dt \Big)
+\lim_{k\rightarrow\infty}E\int_0^T\langle\widehat{\mathcal{A}}(u^{n_k}(s),s),\int_s^T\varphi(t)\,dt\ w\rangle\,ds \\
&+\lim_{k\rightarrow\infty}\int_0^TE\Big(\int_0^t\widehat{\sigma}(u^{n_k}(s),s)\,dW(s),\varphi(t)w\Big)\,dt \\
&+\lim_{k\rightarrow\infty}\int_0^TE\Big(\int_0^t\int_Z\widehat{f}(u^{n_k}(s-),s,z)\widetilde{N}(dzds),\varphi(t)w\Big)\,dt \\
=&E\int_0^T\langle u(0),\varphi(t)w\rangle\,dt +E\int_0^T\int_0^t\langle\widehat{\mathcal{A}}^*(s),\varphi(t)w\rangle\,dsdt \\
&+E\int_0^T\langle\int_0^t\widehat{\sigma}^*(s)\,dW(s),\varphi(t)w\rangle\,dt +E\int_0^T\langle\int_0^t\int_Z\widehat{f}^*(s)\widetilde{N}(dzds),\varphi(t)w\rangle\,dt\\
=&E\int_0^T\langle u(t),\varphi(t)w\rangle\,dt .
\end{align*}

\noindent So, we have $u=\overline{u}$ $dt\otimes P$-a.e. and then $u\in L^4(\Omega;L^{\infty}([0,T];\mathbb{W}))$. By Theorem 2 of \cite{1982-Gyoengy-p153-173} we further deduce that $u$ is a $\mathbb{V}$-valued c\`{a}dl\`{a}g $\mathcal{F}_t$-adapted process and
\begin{align}\label{26.1}
\begin{aligned}
|u(t)|_{\mathbb{V}}^2=&|u(0)|_{\mathbb{V}}^2 +2\int_0^t\langle\widehat{\mathcal{A}}^*(s),\overline{u}(s)\rangle\,ds +\int_0^t|\widehat{\sigma}^*(s)|_{\mathbb{V}}^2\,ds +2\int_0^t(u(s),\widehat{\sigma}^*(s))_{\mathbb{V}}\,dW(s) \\
&+2\int_0^t\int_Z(u(s-),\widehat{f}^*(s,z))_{\mathbb{V}}\widetilde{N}(dzds) +\int_0^t\int_Z|\widehat{f}^*(s,z)|_{\mathbb{V}}^2\,N(dzds) .
\end{aligned}
\end{align}

\noindent Moreover, By the same arguments as in \cite{2001-Temam-p-}, we can show that the paths of $u$ are $\mathbb{W}$-valued weakly  c\`{a}dl\`{a}g. Therefore, in order to prove that $u$ is a solution of equation (\ref{Abstract}), it suffices to prove that \
\begin{align*}
&\widehat{\mathcal{A}}(u(s),s)=\widehat{\mathcal{A}}^*(s),\quad \widehat{\sigma}(u(s),s)=\widehat{\sigma}^*(s),\quad dt\otimes P-a.e.;\\
&\widehat{f}(u(s-),s,z)=\widehat{f}^*(s,z)\quad dt\otimes P\otimes\nu-a.e. .
\end{align*}

Recall $\rho(u,t)$ in (\ref{4.1}).
Let $\phi$ be a $\mathbb{V}$-valued progressively measurable process and $\phi\in L^4(\Omega;L^{\infty}([0,T];\mathbb{W}))$, applying It\^{o} formula (cf.\cite{1982-Metivier-p-}) we have
\begin{align*}
&{\rm e}^{-\int_0^t\rho(\phi(s),s)\,ds}|u^{n_k}(t)|_{\mathbb{V}}^2 \\
=&|u^{n_k}(0)|_{\mathbb{V}}^2+ \int_0^t{\rm e}^{-\int_0^s\rho(\phi(r),r)\,dr}\Big[2\langle\widehat{\mathcal{A}}(u^{n_k}(s),s),u^{n_k}(s)\rangle \\
& +|\Pi_{n_k}\widehat{\sigma}(u^{n_k}(s),s)|_{\mathbb{V}}^2 -\rho(\phi(s),s)|u^{n_k}(s)|_{\mathbb{V}}^2\Big]\,ds\\
& +2\int_0^t{\rm e}^{-\int_0^s\rho(\phi(r),r)\,dr}(\widehat{\sigma}(u^{n_k}(s),s),u^{n_k}(s))_{\mathbb{V}}\,dW(s)\\
& +2\int_0^t\int_Z{\rm e}^{-\int_0^s\rho(\phi(r),r)\,dr}(\widehat{f}(u^{n_k}(s-),s,z),u^{n_k}(s-))_{\mathbb{V}}\widetilde{N}(dzds) \\
& +2\int_0^t\int_Z{\rm e}^{-\int_0^s\rho(\phi(r),r)\,dr}|\Pi_{n_k}\widehat{f}(u^{n_k}(s-),s,z)|_{\mathbb{V}}^2\,N(dzds) .
\end{align*}

\noindent Taking expectations on both sides of the above equality and by the local monotonicity in lemma \ref{Lemma 3.I} we get
\begin{align*}
E\Big({\rm e}^{-\int_0^t\rho(\phi(s),s)\,ds}|u^{n_k}(t)|_{\mathbb{V}}^2\Big) -E|u^{n_k}(0)|_{\mathbb{V}}^2
\leq EI_1(t) +EI_2(t) \leq EI_2(t) ,
\end{align*}
where
\begin{align*}
I_1(t):=&\int_0^t{\rm e}^{-\int_0^s\rho(\phi(r),r)\,dr}\Big(2\langle\widehat{\mathcal{A}}(u^{n_k}(s),s)-\widehat{\mathcal{A}}(\phi(s),s),u^{n_k}(s)-\phi(s)\rangle\\
& +|\widehat{\sigma}(u^{n_k}(s),s)-\widehat{\sigma}(\phi(s),s)|_{\mathbb{V}}^2 +\int_Z|\widehat{f}(u^{n_k}(s),s,z)-\widehat{f}(\phi(s),s,z)|_{\mathbb{V}}^2\,\nu(dz) \\
& -\rho(\phi(s),s)|u^{n_k}(s)-\phi(s)|_{\mathbb{V}}^2\Big)\,ds \leq 0 ;
\end{align*}

\begin{align*}
I_2(t):=&\int_0^t{\rm e}^{-\int_0^s\rho(\phi(r),r)\,dr}\Big(2\langle\widehat{\mathcal{A}}(u^{n_k}(s),s)-\widehat{\mathcal{A}}(\phi(s),s),\phi(s)\rangle\\
&+2\langle\widehat{\mathcal{A}}(\phi(s),s),u^{n_k}(s)\rangle -|\widehat{\sigma}(\phi(s),s)|_{\mathbb{V}}^2 +2\big(\widehat{\sigma}(u^{n_k}(s),s),\widehat{\sigma}(\phi(s),s)\big)_{\mathbb{V}} \\ &+2\int_Z\big(\widehat{f}(u^{n_k}(s),s,z),\widehat{f}(\phi(s),s,z)\big)_{\mathbb{V}}\,\nu(dz) -\int_Z|\widehat{f}(\phi(s),s,z)|_{\mathbb{V}}^2\,\nu(dz) \\
&-2\rho(\phi(s),s)\big(u^{n_k}(s),\phi(s)\big)_{\mathbb{V}} +\rho(\phi(s),s)|\phi(s)|_{\mathbb{V}}^2\Big)\,ds .
\end{align*}

\noindent Since for any nonnegative function $\psi\in L^{\infty}([0,T])$
\begin{align}\label{29.1}
&E\int_0^T\psi(t)|\overline{u}(t)|_{\mathbb{V}}^2\,dt
=\lim_{k\rightarrow\infty}E\int_0^T(\psi(t)\overline{u}(t),u^{n_k}(t))_{\mathbb{V}}\,dt \nonumber\\
\leq &\Big(E\int_0^T\psi(t)|\overline{u}(t)|_{\mathbb{V}}^2\,dt\Big)^{\frac{1}{2}} \liminf_{k\rightarrow\infty}\Big(E\int_0^T\psi(t)|u^{n_k}(t)|_{\mathbb{V}}^2\,dt\Big)^{\frac{1}{2}} ,
\end{align}
noticing that $u=\overline{u}\quad dt\otimes P$-a.e., we have
\begin{align*}
E\int_0^T\psi(t)|u(t)|_{\mathbb{V}}^2\,dt \leq \liminf_{k\rightarrow\infty}E\int_0^T\psi(t)|u^{n_k}(t)|_{\mathbb{V}}^2\,dt .
\end{align*}

\noindent Using the Fubini theorem, then letting $k\rightarrow\infty$ and by Lemma \ref{Lemma 22.I} we obtain that
\begin{align}\label{31.1}
\begin{aligned}
&E\int_0^T\psi(t)\Big[{\rm e}^{-\int_0^t\rho(\phi(s),s)\,ds}|u(t)|_{\mathbb{V}}^2-|u(0)|_{\mathbb{V}}^2\Big]\,dt \\
\leq &\liminf_{k\rightarrow\infty}E\int_0^T\psi(t)\Big[{\rm e}^{-\int_0^t\rho(\phi(s),s)\,ds}|u^{n_k}(t)|_{\mathbb{V}}^2-|u^{n_k}(0)|_{\mathbb{V}}^2\Big]\,dt \\
\leq & \liminf_{k\rightarrow\infty}E\int_0^T\psi(t)I_2(t)\,dt \\
\leq & E\int_0^T\psi(t)\int_0^t{\rm e}^{-\int_0^s\rho(\phi(r),r)\,dr}\Big[2\langle\widehat{\mathcal{A}}^*(s)-\widehat{\mathcal{A}}(\phi(s),s),\phi(s)\rangle\\
&+2\langle\widehat{\mathcal{A}}(\phi(s),s),\overline{u}(s)\rangle -|\widehat{\sigma}(\phi(s),s)|_{\mathbb{V}}^2 +2\big(\widehat{\sigma}^*(s),\widehat{\sigma}(\phi(s),s)\big)_{\mathbb{V}} \\ &+2\int_Z\big(\widehat{f}^*(s,z),\widehat{f}(\phi(s),s,z)\big)_{\mathbb{V}}\,\nu(dz) -\int_Z|\widehat{f}(\phi(s),s,z)|_{\mathbb{V}}^2\,\nu(dz) \\
&-2\rho(\phi(s),s)\big(\overline{u}(s),\phi(s)\big)_{\mathbb{V}} +\rho(\phi(s),s)|\phi(s)|_{\mathbb{V}}^2\Big]\,dsdt .
\end{aligned}
\end{align}

\noindent On the other hand, by (\ref{26.1}) and It\^{o} formula,
\begin{align*}
& E\Big({\rm e}^{-\int_0^t\rho(\phi(s),s)\,ds}|u(t)|_{\mathbb{V}}^2\Big) -E|u(0)|_{\mathbb{V}}^2 \\
=& E\int_0^t{\rm e}^{-\int_0^s\rho(\phi(r),r)\,dr}\Big[2\langle\widehat{\mathcal{A}}^*(s),\overline{u}(s)\rangle
+|\widehat{\sigma}^*(s)|_{\mathbb{V}}^2+\int_Z|\widehat{f}^*(s,z)|_{\mathbb{V}}^2\,\nu(dz)
-\rho(\phi(s),s)|u(s)|_{\mathbb{V}}^2\Big]\,ds .
\end{align*}
Hence, by the Fubini theorem we get
\begin{align}\label{32.1}
\begin{aligned}
&E\int_0^T\psi(t)\Big[{\rm e}^{-\int_0^t\rho(\phi(s),s)\,ds}|u(t)|_{\mathbb{V}}^2-|u(0)|_{\mathbb{V}}^2\Big]\,dt \\
= &E\int_0^T\psi(t)\int_0^t{\rm e}^{-\int_0^s\rho(\phi(r),r)\,dr}\Big[2\langle\widehat{\mathcal{A}}^*(s),\overline{u}(s)\rangle
+|\widehat{\sigma}^*(s)|_{\mathbb{V}}^2 \\
&+\int_Z|\widehat{f}^*(s,z)|_{\mathbb{V}}^2\,\nu(dz)
-\rho(\phi(s),s)|u(s)|_{\mathbb{V}}^2\Big]\,dsdt .
\end{aligned}
\end{align}

\noindent Combining (\ref{31.1}) and (\ref{32.1}) we arrive at
\begin{align}\label{32.2}
\begin{aligned}
&E\int_0^T\psi(t)\int_0^t{\rm e}^{-\int_0^s\rho(\phi(r),r)\,dr}\Big[2\langle\widehat{\mathcal{A}}^*(s)-\widehat{\mathcal{A}}(\phi(s),s),\overline{u}(s)-\phi(s)\rangle +|\widehat{\sigma}^*(s)-\widehat{\sigma}(\phi(s),s)|_{\mathbb{V}}^2 \\ &+\int_Z|\widehat{f}^*(s,z)-\widehat{f}(\phi(s),s,z)|_{\mathbb{V}}^2\,\nu(dz)
 -\rho(\phi(s),s)|\overline{u}(s)-\phi(s)|_{\mathbb{V}}^2\Big]\,dsdt
\leq 0 .
\end{aligned}
\end{align}

\noindent If we put $\phi(s)=\overline{u}(s)$ in (\ref{32.2}), we obtain
\begin{gather*}
\widehat{\sigma}^*(s)=\widehat{\sigma}(\overline{u}(s),s)=\widehat{\sigma}(u(s),s)\quad \text{in}\ \ L^2(\Omega\times[0,T];\mathbb{V});\\
\widehat{f}^*(s,z)=\widehat{f}(\overline{u}(s),s,z)=\widehat{f}(u(s),s,z)=\widehat{f}(u(s-),s,z)\ \text{in}\  L^2(\Omega\times[0,T]\times Z;P\times \nu\times dt;\mathbb{V}).
\end{gather*}

\noindent Thus, (\ref{32.2}) implies that
\begin{align}\label{33.1}
\begin{aligned}
&E\int_0^T\psi(t)\int_0^t{\rm e}^{-\int_0^s\rho(\phi(r),r)\,dr}\Big[2\langle\widehat{\mathcal{A}}^*(s)-\widehat{\mathcal{A}}(\phi(s),s),\overline{u}(s)-\phi(s)\rangle\\
& -\rho(\phi(s),s)|\overline{u}(s)-\phi(s)|_{\mathbb{V}}^2\Big]\,dsdt
\leq 0 .
\end{aligned}
\end{align}

\noindent Putting $\phi=\overline{u}-\varepsilon\widetilde{\phi}w$ in (\ref{33.1}) for $\widetilde{\phi}\in L^{\infty}([0,T]\times\Omega;dt\otimes P)$ and $w\in\mathbb{W}$, then dividing both sides by $\varepsilon$ and letting $\varepsilon\rightarrow 0$, we finally have
\begin{align*}
E\int_0^T\psi(t)\int_0^t{\rm e}^{-\int_0^s\rho(\overline{u}(r),r)\,dr}\Big[2\widetilde{\phi}(s)\langle\widehat{\mathcal{A}}^*(s)-\widehat{\mathcal{A}}(\overline{u}(s),s),w\rangle\Big]\,dsdt
\leq 0 .
\end{align*}
Since $\widetilde{\phi}$ is arbitrary,
\[\widehat{\mathcal{A}}^*(s)=\widehat{\mathcal{A}}(\overline{u}(s),s)=\widehat{\mathcal{A}}(u(s),s)\quad\text{in}\ \ L^2(\Omega\times[0,T];\mathbb{W}^*) .\]

\noindent Therefore, we conclude that the process $u$ is a solution of equation (\ref{Abstract}).

\end{proof}

\subsection{Uniqueness of solutions}

We finally proceed to show the uniqueness of solutions to equation (\ref{Abstract}), thus completes the proof of Theorem \ref{theorem 2.1}.

Suppose that $u=\{u(t)\}$ and $v=\{v(t)\}$ are two solutions of (\ref{Abstract}) with initial conditions $u_0, v_0$ respectively, i.e.
\begin{gather*}
u(t)=u_0+\int_0^t\widehat{\mathcal{A}}(u(s),s)\,ds +\int_0^t\widehat{\sigma}(u(s),s)\,dW(s) +\int_0^t\int_Z\widehat{f}(u(s-),s,z)\widetilde{N}(dsdz),\ t\in[0,T];\\
v(t)=u_0+\int_0^t\widehat{\mathcal{A}}(v(s),s)\,ds +\int_0^t\widehat{\sigma}(v(s),s)\,dW(s) +\int_0^t\int_Z\widehat{f}(v(s-),s,z)\widetilde{N}(dsdz),\ t\in[0,T].
\end{gather*}

\noindent We define the stopping time
\[\tau_N:=\inf\{t\in[0,T]: |u(t)|_{\mathbb{V}}\geq N\}\wedge\inf\{t\in[0,T]: |v(t)|_{\mathbb{V}}\geq N\}\wedge T .\]
Using the It\^{o} formula (cf.\cite{1982-Metivier-p-}) we have
\begin{align*}
&{\rm e}^{-\int_0^{t\wedge\tau_N}\rho(v(s),s)\,ds}|u(t\wedge\tau_N)-v(t\wedge\tau_N)|_{\mathbb{V}}^2 \\
=&|u_0-v_0|_{\mathbb{V}}^2+ \int_0^{t\wedge\tau_N}{\rm e}^{-\int_0^s\rho(v(r),r)\,dr}\Big[2\langle\widehat{\mathcal{A}}(u(s),s)-\widehat{\mathcal{A}}(v(s),s),u(s)-v(s)\rangle \\
& +|\widehat{\sigma}(u(s),s)-\widehat{\sigma}(v(s),s)|_{\mathbb{V}}^2 -\rho(v(s),s)|u(s)-v(s)|_{\mathbb{V}}^2\Big]\,ds\\
& +2\int_0^{t\wedge\tau_N}{\rm e}^{-\int_0^s\rho(v(r),r)\,dr}(\widehat{\sigma}(u(s),s)-\widehat{\sigma}(v(s),s),u(s)-v(s))_{\mathbb{V}}\,dW(s)\\
& +2\int_0^{t\wedge\tau_N}\int_Z{\rm e}^{-\int_0^s\rho(v(r),r)\,dr}(\widehat{f}(u(s-),s,z)-\widehat{f}(v(s-),s,z),u(s-)-v(s-))_{\mathbb{V}}\widetilde{N}(dzds) \\
& +2\int_0^{t\wedge\tau_N}\int_Z{\rm e}^{-\int_0^s\rho(v(r),r)\,dr}|\widehat{f}(u(s-),s,z)-\widehat{f}(v(s-),s,z)|_{\mathbb{V}}^2\,N(dzds) .
\end{align*}

\noindent It then follows from the local monotonicity (H2) of Lemma \ref{Lemma 3.I} that
\begin{align*}
&E\big[{\rm e}^{-\int_0^{t\wedge\tau_N}\rho(v(s),s)\,ds}|u(t\wedge\tau_N)-v(t\wedge\tau_N)|_{\mathbb{V}}^2\Big] -E|u_0-v_0|_{\mathbb{V}}^2 \\
= & E\Big[\int_0^{t\wedge\tau_N}{\rm e}^{-\int_0^s\rho(v(r),r)\,dr}\Big(2\langle\widehat{\mathcal{A}}(u(s),s)-\widehat{\mathcal{A}}(v(s),s),u(s)-v(s)\rangle \\
& +|\widehat{\sigma}(u(s),s)-\widehat{\sigma}(v(s),s)|_{\mathbb{V}}^2 -\rho(v(s),s)|u(s)-v(s)|_{\mathbb{V}}^2 \\
& +\int_Z{\rm e}^{-\int_0^s\rho(v(r),r)\,dr}|\widehat{f}(u(s),s,z)-\widehat{f}(v(s),s,z)|_{\mathbb{V}}^2\,\nu(dz)\Big)\,ds \Big] \\
\leq & 0 .
\end{align*}

\noindent Hence if $u_0=v_0$ $P$-a.e., then
\[E\Big[{\rm e}^{-\int_0^{t\wedge\tau_N}\rho(v(s),s)\,ds}|u(t\wedge\tau_N)-v(t\wedge\tau_N)|_{\mathbb{V}}^2\Big]=0 ,\quad t\in[0,T] .\]

\noindent By (\ref{4.1}) we get
\[\int_0^T\rho(v(s),s)\,ds <\infty,\quad  P\text{-a.s.} .\]

\noindent Therefore, by letting $N\rightarrow\infty$ we have $u(t)=v(t)$ $P$-a.s.. The pathwise uniqueness follows from the c\`{a}dl\`{a}g property of $u,v$ in the solutions.

This completes the proof of Theorem \ref{theorem 2.1}.

\vskip0.5cm {\small {\bf  Acknowledgements}\ \  This work is partly supported by National Natural Science Foundation of China (11671372, 11431014, 11401557)}.


\end{document}